\documentclass[11p4,a4paper]{article}
        
\usepackage{amsmath}
\usepackage{amsthm}
\usepackage{cain_arxiv}
\usepackage{amssymb}
\usepackage{graphicx}

\usepackage{tikz}
\usetikzlibrary{calc,matrix}

\theoremstyle{definition}
\newtheorem{definition}{Definition}[section]

\newtheorem{remark}[definition]{Remark}
\newtheorem{algorithm}[definition]{Algorithm}

\theoremstyle{plain}
\newtheorem{proposition}[definition]{Proposition}
\newtheorem{lemma}[definition]{Lemma}
\newtheorem{theorem}[definition]{Theorem}
\newtheorem{corollary}[definition]{Corollary}

\numberwithin{equation}{section}

\def\fullref#1#2{%
    {#1\space\penalty 200\relax\ref{#2}}%
}

\newcommand{\defterm}[1]{\textit{#1}}

\newcommand{\rpad}{\delta_{\mathrm{R}}}
\newcommand{\lpad}{\delta_{\mathrm{L}}}

\newcommand{\gen}[1]{\left\langle #1 \right\rangle}

\newcommand{\pres}[2]{\left\langle #1\:|\:#2 \right\rangle}

\newcommand{\nset}{\mathbb{N}}

\newcommand{\emptyword}{\varepsilon}

\newcommand{\rel}[1]{\mathcal{#1}}

\newcommand{\imreduces}{\rightarrow}

\newcommand{\rev}{\mathrm{rev}}

\begin{document}

\title{Finite Gr\"{o}bner--Shirshov bases for Plactic algebras and biautomatic structures for Plactic monoids}
\author{Alan J.Cain, Robert D. Gray, Ant\'onio Malheiro}

\thanks{The first author was supported by the European Regional Development Fund through the programme {\sc COMPETE} and
  by the Portuguese Government through the {\sc FCT} (Funda\c{c}\~{a}o para a Ci\^{e}ncia e a Tecnologia) under the
  project {\sc PEst-C}/{\sc MAT}/{\sc UI}0144/2011 and through an {\sc FCT} Ci\^{e}ncia 2008 fellowship. For the second
  and third author, this work was supported by CAUL within the project {\sc PEst-OE}/{\sc MAT}/{\sc UI}0143/2012--13
  financed by Funda\c{c}\~{a}o para a Ci\^{e}ncia e a Tecnologia.}

\maketitle

\address[AJC]{%
Centro de Matem\'{a}tica, Faculdade de Ci\^encias, Universidade do Porto\\
Rua do Campo Alegre 687, 4169--007 Porto, Portugal
}
\email{%
ajcain@fc.up.pt
}
\webpage{%
www.fc.up.pt/pessoas/ajcain/
}

\address[RDG]{%
Centro de \'{A}lgebra da Universidade de Lisboa\\
Av. Prof. Gama Pinto 2, 1649--003 Lisboa, Portugal
}
\email{Robert.D.Gray@uea.ac.uk}

\address[AM]{%
Centro de \'{A}lgebra da Universidade de Lisboa\\
Av. Prof. Gama Pinto 2, 1649--003 Lisboa, Portugal \\
\null\quad and \\
Departamento de Matem\'{a}tica,\\
Faculdade de Ci\^{e}ncias e Tecnologia da Universidade Nova de Lisboa,\\
2829--516 Caparica, Portugal
}
\email{ajm@fct.unl.pt}

\begin{abstract}
This paper shows that every Plactic algebra of finite rank admits a
finite Gr\"{o}bner--Shirshov basis. The result is proved by using the
combinatorial properties of Young tableaux to construct a finite
complete rewriting system for the corresponding Plactic monoid, which
also yields the corollaries that Plactic monoids of finite rank have
finite derivation type and satisfy the homological finiteness properties 
left and right $\mathrm{FP}_\infty$. Also, answering a question of Zelmanov, 
we apply this rewriting system and other techniques to show that Plactic monoids 
of finite rank are biautomatic. 
\end{abstract}

\section{Introduction}

The Plactic monoid has its origins in work of Schensted
\cite{schensted_longest} and Knuth \cite{knuth_permutations} concerned
with certain combinatorial problems and operations on Young
tableaux. It was later studied in depth by Lascoux and Sh\"utzenberger
\cite{lascoux_plaxique} and has since become an important tool in several
aspects of representation theory and algebraic combinatorics; see
\cite{fulton_young,lothaire_algebraic}.  The first significant application of the
Plactic monoid was to the Littlewood--Richardson rule for Schur
functions. This is explained in detail in the appendix to the second
edition of J. A. Green's influential monograph on the representation
theory of the general linear group \cite{green_polynomial}. The
Littlewood--Richardson rule \cite{littlewood_group} is one of the most
important results in the theory of symmetric functions. It provides a
combinatorial rule for expressing a product of two Schur functions as
a linear combination of Schur functions. Since Schur functions in $n$
variables are the irreducible polynomial characters of
$GL_n(\mathbb{C})$, the Littlewood-Richardson rule gives a tensor
product rule for $GL_n(\mathbb{C})$.  One of the most enlightening
proofs of the Littlewood--Richardson rule (see
\cite[Section~5.4]{lothaire_algebraic}) is given by lifting the calculus of the
Schur function to the integral monoid ring of the Plactic monoid
(called the tableau ring; see \cite[Chapter~2]{fulton_young}).

Subsequently the Plactic monoid has been found to have applications in
a range of areas including a combinatorial description of
Kostka--Foulkes polynomials \cite{lascoux_plaxique,lascoux_foulkes}, and to Kashiwara's
theory of crystal bases \cite{date_representations,kashiwara_crystalbases} leading to the
definition of Plactic algebras associated to all classical simple Lie
algebras \cite{littelmann_plactic,lascoux_crystal,krob_noncommutative}.  Further results on
Robinson--Schensted correspondence and the Plactic relations may be
found in \cite{date_representations,leclerc_robinson}.  Several variations and generalizations
of the Plactic monoid have been proposed and investigated including
hypoplactic monoids \cite{krob_noncommutative}, and shifted Plactic monoids
\cite{serrano_shifted}. In \cite{duchamp_plactic} it is show that the Hilbert series of
the Plactic monoid is given by the Schur--Littlewood formula, and that
there are exactly three families of ternary monoids with this Hilbert
series. Sch\"utzenberger \cite{schutzenberger_pour} argues that the
Plactic monoid ought to be considered as ``one of the most fundamental
monoids in algebra''. He cites three reasons for his own personal
``weakness'' for the Plactic monoid, the first of them being the
application to symmetric functions mentioned above.

Various aspects of the corresponding semigroup algebras, the Plactic
algebras, have been investigated; see, for example,
\cite{cedo_plactic,lascoux_keys}.  These algebras are important
special cases in the more general study of algebras
defined by homogeneous semigroup presentations \cite{cedo_finitelypresented}.
Frequently, fundamental problems about such semigroup algebras require
detailed analysis of the corresponding semigroups.  An important
example of this is given by the theory of Gr\"{o}bner--Shirshov bases.
Kubat \& Okni\'{n}ski showed that the Plactic algebra of rank $3$ has
a finite Gr\"{o}bner--Shirshov basis \cite[Theorem~1]{kubat_grobner}
and that Plactic algebras of rank $4$ or more do not admit a finite
Gr\"{o}bner--Shirshov basis with respect to the degree-lexicographic
ordering over the usual generating set for the Plactic monoid
\cite[Theorem~3]{kubat_grobner}. In contrast, the related Chinese
monoid admits a finite complete rewriting system with respect to the usual
generating set \cite{guzelkarpuz_fcrs}, and so its semigroup algebra,
the Chinese algebra, is known to admit a finite Gr\"{o}bner--Shirshov
basis \cite{chen_grobner}.

The first aim of this paper is to use the combinatorial properties
of Young tableaux to construct finite complete rewriting systems for
Plactic monoids of arbitrary finite rank, and thus prove that the
corresponding Plactic algebras admit finite Gr\"{o}bner--Shirshov
bases (see \cite{heyworth_rewriting} for an explanation of the
connection between Gr\"obner--Shirshov bases and complete rewriting
systems).  The rewriting system is \emph{not} over the usual
generating set for the Plactic monoid; rather, the generating set
comprises the (finite) set of columns of Young tableaux.  As a
corollary we deduce that Plactic monoids of finite rank satisfy the
homological finiteness property $\mathrm{FP}_\infty$, a result which
gives information about the existence of free resolutions of
$\mathbb{Z}M_n$-modules, where $\mathbb{Z}M_n$ is the tableau ring
featuring in the theory of symmetric functions outlined above.

During the writing of this paper, the authors came across the work of
Chen \& Li \cite{chen_new}, who exhibit infinite complete rewriting
systems for Plactic monoids over the (infinite) set of rows of Young
tableaux. Thus Chen \& Li's work yields infinite Gr\"{o}bner--Shirshov
bases for Plactic algebras. Part of their reasoning is an analogue for
rows of \fullref{Lemma}{lem:incompcols} below, but they use a direct,
more technical, proof and later recover as a corollary of their main
result the fact that tableaux form a cross-section of the Plactic
monoid.

As a consequence of the Schensted insertion algorithm and the
representation of elements by tableaux, it follows that the Plactic
monoid has word problem that is solvable in quadratic time.  This
leads us naturally to the second major theme of the present article:
the subject of automatic structures.  The concept of an automatic
group was introduced in order to describe a large class of groups with
easily solvable word problem.  The best general reference for the
theory of automatic groups is the book \cite{epstein_wordproc}.  The
notion has been extended to automatic monoids and semigroups
\cite{campbell_autsg}.  In both cases the defining property is the
existence of a rational set of normal forms (with respect to some
finite generating set $A$) such that we have, for each generator in
$A$, a finite automaton that recognizes pairs of normal forms that
differ by multiplication by that generator. It is a consequence of the
definition that automatic monoids (and in particular automatic groups)
have word problem that is solvable in quadratic time
\cite[Corollary~3.7]{campbell_autsg}.

Automatic groups have attracted a lot of attention over the last $20$
years, in part because of the large number of natural and important
classes of groups that have this property.  The class of automatic
groups includes: finite groups, free groups, free abelian groups,
various small cancellation groups \cite{gersten_smallcancellation},
Artin groups of finite and large type \cite{holt_artingroups}, Braid
groups, and hyperbolic groups in the sense of Gromov
\cite{gromov_hyperbolic}. In parallel, the theory of automatic monoids
has been extended and developed over recent years.  Classes of monoids
that have been shown to be automatic include divisibility monoids
\cite{picantin_finite} and singular Artin monoids of finite type
\cite{corran_singular}.  Several complexity and decidability results
for automatic monoids are obtained in \cite{lohrey_decidability}.
Other aspects of the theory of automatic monoids that have been
investigated include connections with the theory of Dehn functions
\cite{otto_dehn} and complete rewriting systems
\cite{otto_automonversus}.

Given the algorithmic properties of the Plactic monoid mentioned
above, the natural question of whether the Plactic monoid itself
admits an automatic structure was asked by Efim Zelmanov [during his
  plenary lecture at the international conference \emph{Groups and
    Semigroups: Interactions and Computations} (Lisbon, 25--29 July
  2011)]. The second main result of this article is an affirmative
answer to this question.  Beginning with the finite complete rewriting
system obtained in Section~\ref{sec_CRS}, we shall show how for
Plactic monoids finite transducers may be constructed to perform left
(respectively right) multiplication by a generator. We then apply this
result to show that Plactic monoids of arbitrary finite rank are
biautomatic (the strongest form of automaticity for monoids).

\section{Preliminaries}

This paper assumes familiarity with rewriting systems,
Gr\"{o}bner--Shirshov bases, automata and regular languages, and
transducers and rational relations.

For background information, see, for example, \cite{book_srs} on
complete rewriting systems; \cite{ufnarovski_introduction} on
Gr\"{o}bner--Shirshov bases; \cite{heyworth_rewriting} on the
connection between them. See also \cite{hopcroft_automata} on automata
and regular languages and \cite{berstel_transductions}) on transducers
and rational relations.

We denote the empty word (over any alphabet) by $\emptyword$. For an
alphabet $A$, we denote by $A^*$ the set of all words over $A$. When
$A$ is a generating set for a monoid $M$, every element of $A^*$ can
be interpreted either as a word or as an element of $M$. For words
$u,v \in A^*$, we write $u=v$ to indicate that $u$ and $v$ are equal
as words and $u=_M v$ to denote that $u$ and $v$ represent the same
element of the monoid $M$. The length of $u \in A^*$ is denoted
$|u|$. For a relation $\rel{R}$ on $A^*$, the presentation
$\pres{A}{\rel{R}}$ defines [any monoid isomorphic to]
$A^*/\rel{R}^\#$, where $\rel{R}^\#$ denotes the congruence generated
by $\rel{R}$.

\subsection{Plactic monoid}

This section recalls only the relevant definition and properties of
the Plactic monoid; for a full introduction, see
\cite[Chapter~5]{lothaire_algebraic}.

Let $n \in \nset$. Let $A$ be the finite ordered alphabet $\{1 < 2 <
\ldots < n\}$. Let $\rel{R}$ be the set of defining relations
\begin{equation}
\label{eq:placticrel}
\{(xzy,zxy) : x \leq y < z\} \cup \{(yxz,yzx) : x<y\leq z\}.
\end{equation}
Then the \defterm{Plactic monoid} $M_n$ is presented by
$\pres{A}{\rel{R}}$.

A \defterm{row} is a non-decreasing word in $A^*$ (that is, a word
$\alpha = \alpha_1\cdots \alpha_k$, where $\alpha_i \in A$, in which
$\alpha_i \leq \alpha_{i+1}$ for all $i =1,\ldots,k-1$). Let $\alpha =
\alpha_1\cdots \alpha_k$ and $\beta = \beta_1\cdots\beta_l$ (where
$\alpha_i,\beta_i \in A$) be rows. The row $\alpha$
\defterm{dominates} the row $\beta$, denoted $\alpha \triangleright
\beta$, if $k \leq l$ and $\alpha_i > \beta_i$ for all $i =
1,\ldots,k$.

Any word $w \in A^*$ has a decomposition as a product of rows of
maximal length $w = \alpha^{(1)}\cdots\alpha^{(k)}$. Such a word $w$
is a \defterm{tableau} if $\alpha^{(i)} \triangleright \alpha^{(i+1)}$
for all $i = 1,\ldots,k-1$. It is usual to write tableaux in a planar
form, with the rows placed in order of domination and
left-justified. For example, the tableau $6\;3455\;11235$ is written as follows:
\[
\begin{tikzpicture}
\matrix [matrix of math nodes,nodes={rectangle,draw,minimum width=6mm,minimum height=6mm},row sep={between borders,-\pgflinewidth},column sep={between borders,-\pgflinewidth}]
{
6 \\
3 & 4 & 5 & 5 \\
1 & 1 & 2 & 3 & 5 \\
};
\end{tikzpicture}
\]
The set of tableaux form a cross-section of the Plactic monoid $M_n$
\cite[Theorem~5.2.5]{lothaire_algebraic}. For each $u \in A^*$, denote
by $P(u)$ the unique tableau with $P(u) =_{M_n} u$. If $u$ is a tableau,
$P(u) = u$. Since the defining relations in the presentation
$\pres{A}{\rel{R}}$ preserve the number of symbols, it follows that
$|P(u)| = |u|$ for all $u \in A^*$.

A \defterm{column} is a strictly decreasing word in $A^*$ (that is, a
word $\alpha = \alpha_k\cdots \alpha_1$, where $\alpha_i \in A$, in which
$\alpha_{i+1} > \alpha_{i}$ for all $i =1,\ldots,k-1$). [Notice the
  decreasing order of the subscripts on symbols of columns, so as to
  match the order of the symbols themselves.] This definition matches
the notion of a column in the planar representation of a tableau.

Define a relation $\succeq$ on columns as follows: if $\alpha =
\alpha_k\cdots\alpha_1$ and $\beta = \beta_l\cdots\beta_1$, then
$\alpha \succeq \beta$ if and only if $k \geq l$ and $\alpha_i \leq
\beta_i$ for all $i \leq l$. Thus $\alpha \succeq \beta$ if and only if
the column $\alpha$ can appear immediately to the left of $\beta$ in
the planar representation of a tableau.

For any tableau $w$, denote by $C(w)$ the word obtained by reading
(the planar representation of) that tableau column-wise from left to
right and top to bottom. In the example above, $C(6\;3455\;11235) =
631\;41\;52\;53\;5$. Then $C(w) =_{M_n} w$ for all tableau $w$
\cite[Problem~5.2.4]{lothaire_algebraic}. 

The following result states the key combinatorial facts about
tableaux:

\begin{theorem}[{\cite[Theorems~1 \&~2]{schensted_longest}; see also \cite[Theorem~5.1.1]{lothaire_algebraic}}]
\label{thm:schensted}
Let $u \in A^*$. The number of columns in $P(u)$ is equal to
the length of the longest non-decreasing subsequence in $u$. The
number of rows in $P(u)$ is equal to the length of the longest
decreasing subsequence in $u$.
\end{theorem}

Let $w$ be a tableau and let $\gamma \in A$. The unique tableau
$P(w\gamma)$ equal to $w\gamma$ in $M_n$ can be computed via
Schensted's algorithm \cite[\S~5.1--2]{lothaire_algebraic}, which we
recall here:

\begin{algorithm}[Schensted's algorithm]
\label{alg:schensted}
~\par\nobreak
\textit{Input:} A tableau $w$ with rows
$\alpha^{(1)},\ldots,\alpha^{(k)}$ and a symbol $\gamma \in A$.

\textit{Output:} The unique tableau $P(w\gamma)$ equal to $w\gamma$
in $M_n$.

\textit{Method:}
\begin{enumerate}

\item If $\alpha^{(k)}\gamma$ is a row, the result is
$\alpha^{(1)}\cdots\alpha^{(k)}\gamma$.

\item If $\alpha^{(k)}\gamma$ is not
a row, then suppose $\alpha^{(k)} = \alpha_1\cdots \alpha_l$ (where
$\alpha_i \in A$) and let $j$ be minimal such that $\alpha_j >
\gamma$. Then the result is
$P(\alpha^{(1)}\cdots\alpha^{(k-1)}\alpha_j)\alpha'^{(k)}$, where
$\alpha'^{(k)} = \alpha_1\cdots
\alpha_{j-1}\gamma\alpha_{j+1}\cdots\alpha_l$.

\end{enumerate}
\end{algorithm}
Notice that in case~2, the algorithm replaces $\alpha_j$ by $\gamma$
in the lowest row and recursively right-multiplies by $\alpha_j$ the
tableau formed by all rows except the lowest. This is referred to as
`bumping' $\alpha_j$ to a higher row. When $\alpha_j$ is bumped, it
will be inserted into the row above either in the same column or in
some column further to the left, as shown in
\fullref{Figure}{fig:bump}. This happens because columns are strictly
decreasing from top to bottom, so either the cell above $\alpha_j$
contains some symbol $\eta$ greater than $\alpha_j$, or $\alpha_j$ is
the topmost element of its column. In the former case, $\alpha_j$ will
be inserted so as to replace the leftmost symbol greater than
$\alpha_j$, which must either be to the left of $\eta$ or $\eta$
itself, since rows are non-decreasing from left to right. In the
latter case, $\alpha_j$ will be appended to the end of the row above
and so will be placed either in the same column or further left.

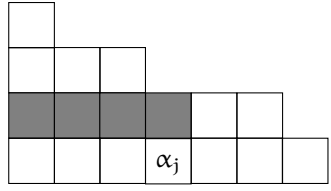
\begin{figure}[t]
\centering
\begin{tikzpicture}
\matrix [matrix of math nodes,nodes={rectangle,draw,minimum width=6mm,minimum height=6mm},row sep={between borders,-\pgflinewidth},column sep={between borders,-\pgflinewidth}]
{
\mathstrut \\
\mathstrut & \mathstrut & \mathstrut \\
|[fill=gray]| \mathstrut & |[fill=gray]| \mathstrut & |[fill=gray]| \mathstrut & |[fill=gray]| \mathstrut & \mathstrut & \mathstrut \\
\mathstrut & \mathstrut & \mathstrut & \mathstrut\alpha_j & \mathstrut & \mathstrut & \mathstrut \\
};
\end{tikzpicture}
\caption{If the symbol $\alpha_j$ is bumped during Schensted's
  algorithm, it must be inserted into one of the shaded cells, since
  the cell above $\alpha_j$ must contain a symbol strictly greater than
  $\alpha_j$.}
\label{fig:bump}
\end{figure}

For any word $u \in A^*$, the tableau $P(u)$ can be effectively
computed by starting with the empty word, which is a valid tableau,
and iteratively applying Schensted's algorithm.

\subsection{Biautomatic structures}

This subsection contains the definitions and basic results from the
theory of automatic and biautomatic monoids needed hereafter. For
further information on automatic semigroups,
see~\cite{campbell_autsg}. 

\begin{definition}
Let $A$ be an alphabet and let $\$$ be a new symbol not in $A$. Define
the mapping $\rpad : A^* \times A^* \to ((A\cup\{\$\})\times (A\cup
\{\$\}))^*$ by
\[
(u_1\cdots u_m,v_1\cdots v_n) \mapsto
\begin{cases}
(u_1,v_1)\cdots(u_m,v_n) & \text{if }m=n,\\
(u_1,v_1)\cdots(u_n,v_n)(u_{n+1},\$)\cdots(u_m,\$) & \text{if }m>n,\\
(u_1,v_1)\cdots(u_m,v_m)(\$,v_{m+1})\cdots(\$,v_n) & \text{if }m<n,
\end{cases}
\]
and the mapping $\lpad : A^* \times A^* \to ((A\cup\{\$\})\times (A\cup \{\$\}))^*$ by
\[
(u_1\cdots u_m,v_1\cdots v_n) \mapsto
\begin{cases}
(u_1,v_1)\cdots(u_m,v_n) & \text{if }m=n,\\
(u_1,\$)\cdots(u_{m-n},\$)(u_{m-n+1},v_1)\cdots(u_m,v_n) & \text{if }m>n,\\
(\$,v_1)\cdots(\$,v_{n-m})(u_1,v_{n-m+1})\cdots(u_m,v_n) & \text{if }m<n,
\end{cases}
\]
where $u_i,v_i \in A$.
\end{definition}

\begin{definition}
\label{def:autstruct}
Let $M$ be a monoid. Let $A$ be a finite alphabet representing a set
of generators for $M$ and let $L \subseteq A^*$ be a regular language such
that every element of $M$ has at least one representative in $L$.  For
each $a \in A \cup \{\emptyword\}$, define the relations
\begin{align*}
L_a &= \{(u,v): u,v \in L, {ua} =_M {v}\}\\
{}_aL &= \{(u,v) : u,v \in L, {au} =_M {v}\}.
\end{align*}
The pair $(A,L)$ is an \defterm{automatic structure} for $M$ if
$L_a\rpad$ is a regular languages over $(A\cup\{\$\}) \times
(A\cup\{\$\})$ for all $a \in A \cup \{\emptyword\}$. A monoid $M$ is
\defterm{automatic} if it admits a automatic structure with respect to
some generating set.

The pair $(A,L)$ is a \defterm{biautomatic structure} for $M$ if
$L_a\rpad$, ${}_aL\rpad$, $L_a\lpad$, and ${}_aL\lpad$ are regular
languages over $(A\cup\{\$\}) \times (A\cup\{\$\})$ for all $a \in A
\cup \{\emptyword\}$. A monoid $M$ is \defterm{biautomatic} if it
admits a biautomatic structure with respect to some generating
set. [Note that biautomaticity implies automaticity.]
\end{definition}

Unlike the situation for groups, biautomaticity for monoids and
semigroups, like automaticity, is dependent on the choice of
generating set \cite[Example~4.5]{campbell_autsg}. However, for
monoids, biautomaticity and automaticity are independent of the choice
of \emph{semigroup} generating sets \cite[Theorem~1.1]{duncan_change}.

Hoffmann \& Thomas have made a careful study of biautomaticity for
semigroups \cite{hoffmann_biautomatic}. They distinguish four notions
of biautomaticity for semigroups:
\begin{itemize}

\item \defterm{right-biautomaticity}, where $L_a\rpad$ and ${}_aL\rpad$ are
  regular languages;

\item \defterm{left-biautomaticity}, where $L_a\lpad$ and ${}_aL\lpad$ are
  regular languages;

\item \defterm{same-biautomaticity}, where $L_a\rpad$ and ${}_aL\lpad$ are
  regular languages;

\item \defterm{cross-biautomaticity}, where ${}_aL\rpad$ and $L_a\lpad$ are
  regular languages.

\end{itemize}
These notions are all equivalent
for groups and more generally for cancellative semigroups
\cite[Theorem~1]{hoffmann_biautomatic} but distinct for semigroups
\cite[Remark~1 \& \S~4]{hoffmann_biautomatic}. In the sense used in
this paper, `biautomaticity' implies \emph{all four} notions of
biautomaticity above.

In proving certain that $R\rpad$ or $R\lpad$ is regular, where $R$ is
a relation on $A^*$, a useful strategy is to prove that $R$ is a
rational relation (that is, a relation recognized by a finite
transducer \cite[Theorem~6.1]{berstel_transductions}) and then apply
the following result, which is a combination of
\cite[Corollary~2.5]{frougny_synchronized} and
\cite[Proposition~4]{hoffmann_biautomatic}:

\begin{proposition}
\label{prop:rationalbounded}
If $R \subseteq A^* \times A^*$ is rational relation and there is a
constant $k$ such that $\bigl||u|-|v|\bigr| \leq k$ for all $(u,v) \in
R$, then $R\rpad$ and $R\lpad$ are regular.
\end{proposition}

\begin{remark}
When constructing transducers to recognize particular relations, we
will make use of certain strategies.

One strategy will be to consider a transducer reading elements of a
relation $R$ from \emph{right to left}, instead of (as usual) left to
right. In effect, such a transducer recognizes the reverse of $R$,
which is the relation
\[
R^\rev = \{(u^\rev,v^\rev) : (u,v) \in R\},
\] 
where $u^\rev$ and $v^\rev$ are the reverses of the words $u$ and $v$
respectively. Since the class of rational relations is closed under
reversal \cite[p.65--66]{berstel_transductions}, constructing such a
(right-to-left) transducer suffices to show that $R$ is a rational
relation.

Another important strategy will be for the transducer to
non-deterministically guess some symbol yet to be read. More exactly,
the transducer will non-deterministically select a symbol and store it
in its state. When it later reads the relevant symbol, it checks it
against the stored guessed symbol. If the guess was correct, the
transducer continues. If the guess was wrong, the transducer enters a
failure state. Similarly, the transducer can non-deterministically
guess that it has reached the end of its input and enter an accept
state. If it subsequently reads another symbol, it knows that its
guess was wrong, and it enters a failure state.
\end{remark}

\section{Complete rewriting system \& Gr\"{o}bner--Shirshov basis}
\label{sec_CRS}

The aim of this section is to construct a finite complete rewriting
system for $M_n$ and so deduce the existence of a finite
Gr\"{o}bner--Shirshov basis for the the corresponding Plactic algebra.

The following lemma will play a crucial role in defining the
rewriting system:

\begin{lemma}
\label{lem:incompcols}
Suppose $\alpha$ and $\beta$ are columns with $\alpha \not\succeq
\beta$. Then $P(\alpha\beta)$ contains at most two
columns. Furthermore, if $P(\alpha\beta)$ contains exactly two
columns, the left column contains more symbols than $\alpha$.
\end{lemma}

\begin{proof}
Since $\alpha$ and $\beta$ are strictly decreasing, the longest
non-decreasing sequence in $\alpha\beta$ is at most $2$, since it can
contain at most one symbol from each of $\alpha$ and $\beta$. (It may
have length $1$ if every symbol in $\beta$ is less than the minimum
symbol in $\alpha$.) Hence by \fullref{Theorem}{thm:schensted},
$P(\alpha\beta)$ contains at most two columns.

\begin{figure}[t]
\centering
\begin{tikzpicture}
\matrix (notrelated) [matrix of math nodes,nodes={rectangle,draw,minimum width=6mm,minimum height=6mm},row sep={between borders,-\pgflinewidth},column sep={between borders,-\pgflinewidth}]
{
  & 6 \\
4 & 5 \\
2 & 3 \\
1 & 1 \\
};
\draw[->,densely dotted,line width=.7pt] ($(notrelated-1-2.north)+(2mm,3mm)$) -- ($(notrelated-4-2.south)+(2mm,-3mm)$);
\end{tikzpicture}
\qquad
\begin{tikzpicture}[line width=.5pt]
\matrix (notrelated) [matrix of math nodes,nodes={rectangle,draw,minimum width=6mm,minimum height=6mm},row sep={between borders,-\pgflinewidth},column sep={between borders,-\pgflinewidth}]
{
6 &  \\
4 & 5 \\
3 & 2 \\
1 & 1 \\
};
\draw[->,densely dotted,line width=.7pt] ($(notrelated-1-1.north)+(2mm,3mm)$) -- ($(notrelated-3-1.north)+(2mm,-1mm)$) -- ($(notrelated-3-2.north)+(2mm,-1mm)$) -- ($(notrelated-4-2.south)+(2mm,-3mm)$);
\end{tikzpicture}
\caption{Two columns not related by $\succeq$ always contain a decreasing
  subsequence of length greater than the left column, as indicated by
  the dotted arrows.}
\label{fig:incompcols}
\end{figure}
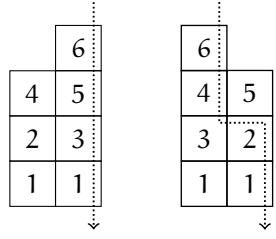

Suppose that $P(\alpha\beta)$ contains exactly two columns. Let
$\alpha = \alpha_k\cdots\alpha_1$ and $\beta = \beta_l\cdots\beta_1$. Then
since $\alpha \not\succeq \beta$, either $k < l$ or $\alpha_i > \beta_i$
for some $i \leq l$, as in the examples in
\fullref{Figure}{fig:incompcols}. In the first case, $\beta$ is a
decreasing subsequence of $\alpha\beta$ containing more symbols than
$\alpha$. In the second case,
$\alpha_k\cdots\alpha_i\beta_i\cdots\beta_1$ is a decreasing
subsequence of $\alpha\beta$ of length $k+1$ and hence contains more
symbols than $\alpha$. In either case, $\alpha\beta$ contains a
decreasing sequence of length greater than $\alpha$, and so by
\fullref{Theorem}{thm:schensted}, $P(\alpha\beta)$ contains more rows
than there are symbols in $\alpha$, and hence the left column of
$P(\alpha\beta)$ contains more symbols than $\alpha$.
\end{proof}

To construct a finite complete rewriting system presenting $M_n$,
introduce a new set of generators. Let
\[
C = \{c_\alpha : \text{$\alpha \in A^+$ is a column}\}
\]
The idea is that each symbol $c_\alpha$ represents the element
$\alpha$ of $M_n$. Thus the symbols $c_1,c_2,\ldots,c_n$ represent the original
generating set for $M_n$, and so the set $C$ also generates
$M_n$. Furthermore, since the set of columns is finite (since a
strictly decreasing sequence of elements of $A$ has length at most
$|A|$), the set $C$ is finite. Notice that $M_n$ is presented by
$\pres{C}{\rel{R}' \cup \rel{S}}$, where
\begin{align*}
\rel{R}' ={}& \{(c_xc_zc_y,c_zc_xc_y) : x,y,z \in A \land x \leq y < z\} \\
&\cup \{(c_yc_xc_z,c_yc_zc_x) : x,y,z \in A \land x < y \leq z\} \\
\rel{S} ={}& \{(c_{\alpha_k\cdots\alpha_1},c_{\alpha_k}\cdots c_{\alpha_1}) : \text{$\alpha_k\cdots\alpha_1$ is a column}\};
\end{align*}
the relations $\rel{R}'$ are simply those in $\rel{R}$ expressed using
the symbols $c_x$ (where $x \in A$), and those in $\rel{S}$ define the extra
generators $C_\alpha$ where $|\alpha| \geq 2$.

Define a set of rewriting rules $\rel{T}$ on $C^*$ as follows:
\begin{align}
\rel{T} ={}& \bigl\{c_\alpha c_\beta \imreduces c_\gamma : \alpha \not\succeq \beta \land \text{$P(\alpha\beta)$ consists of one column $\gamma$}\bigr\} \label{eq:ronecol}\\
\cup{}&\bigl\{c_\alpha c_\beta \imreduces c_\gamma c_\delta : \alpha \not\succeq \beta \land \label{eq:rtwocol}\\
&\qquad\qquad \text{$P(\alpha\beta)$ consists of two columns, left col.~$\gamma$ and right col.~$\delta$}\bigr\}  \nonumber
\end{align}
Notice that every rule in $\rel{T}$ holds in the monoid $M_n$: this
follows from the facts that $c_\zeta =_{M_n} \zeta$ for any column
$\zeta$, that $u =_{M_n} P(u)$ for all $u \in A^*$, and that $C(w)
=_{M_n} w$ for all tableau $w$. For type \eqref{eq:ronecol} rules,
$c_\alpha c_\beta =_{M_n} \alpha\beta =_{M_n} P(\alpha\beta) =_{M_n}
C(P(\alpha\beta)) = \gamma =_{M_n} c_\gamma$; for type
\eqref{eq:rtwocol} rules, $c_\alpha c_\beta =_{M_n} \alpha\beta
=_{M_n} P(\alpha\beta) =_{M_n} C(P(\alpha\beta)) = \gamma\delta
=_{M_n} c_\gamma c_\delta$. Thus every rule in $\rel{T}$ is a
consequence of the relations in $\rel{R}' \cup \rel{S}$.

Notice further that by $k-1$ applications of type \eqref{eq:ronecol}
rules, one can deduce every relation $(c_{\alpha_k\cdots\alpha_1},
c_{\alpha_k}\cdots c_{\alpha_1})$. Finally, it is easy to see that
every relation in $\rel{R}'$ is also a consequence of those in
$\rel{T}$. Thus $M_n$ is presented by $\pres{C}{\rel{T}}$. It remains
to show that $(C,\rel{T})$ is a finite complete rewriting system.

By \fullref{Lemma}{lem:incompcols}, if $\alpha \not\succeq \beta$, then
$P(\alpha\beta)$ has at most two columns. Hence $\rel{T}$ contains a
rewriting rule with left-hand side $c_\alpha c_\beta$ whenever $\alpha
\not\succeq \beta$. Furthermore, since $P(\alpha\beta)$ is uniquely
determined, $\rel{T}$ contains \emph{exactly} one such rewriting rule,
and hence the number of rules in $\rel{T}$ is finite.

\begin{lemma}
\label{lem:placticnoetherian}
The rewriting system $(C,\rel{T})$ is noetherian.
\end{lemma}

\begin{proof}
Choose an ordering $\sqsubset$ on $C$ that reverses the partial order
induced by lengths of subscripts, in the sense that $c_\alpha
\sqsubset c_\beta$ whenever $|\alpha| > |\beta|$. (Such an order must
exist: simply reverse the order induces by the length of subscripts
and then arbitrarily order elements with same-length subscripts.)

Let $\ll$ be the length-plus-lexicographic order on $C^*$
induced by $\sqsubset$. That is:
\begin{multline*}
c^{(1)} c^{(2)}\cdots c^{(k)} \ll d^{(1)}d^{(2)}\cdots d^{(l)} \\
\iff k < l \lor \Bigl(k=l \land (\exists i)\bigl(c^{(i)} \sqsubset d^{(i)} \land (\forall j < i)(c^{(j)} = d^{(j)})\bigr)\Bigr),
\end{multline*}
where all symbols $c^{(h)}$ and $d^{(h)}$ lie in $C$. Then $\ll$ is a
well-ordering of $C^*$. The aim is to prove that if $w \imreduces w'$,
then $w' \ll w$.

First, if the rule applied to obtain $w'$ from $w$ is of type
\eqref{eq:ronecol}, then $w = pc_\alpha c_\beta q$ and $w' = pc_\gamma
q$ for some $p,q \in C^*$ and $c_\alpha,c_\beta,c_\gamma \in C$. So
$w'$ is a shorter word than $w$ and so $w' \ll w$.

Second, if the rule applied to obtain $w'$ from $w$ is of type
\eqref{eq:rtwocol}, then $w = pc_\alpha c_\beta q$ and $w' = pc_\gamma
c_\delta q$ for some $p,q \in C^*$ and
$c_\alpha,c_\beta,c_\gamma,c_\delta \in C$ with $P(\alpha\beta)$
having columns $\gamma$ and $\delta$. By
\fullref{Lemma}{lem:incompcols}, $\gamma$ contains more symbols than
$\alpha$; that is, $|\gamma| > |\alpha|$. Hence, $c_\gamma \sqsubset
c_\alpha$ by the choice of $\sqsubset$. So in the definition of $\ll$,
we have $k=l$ and $c^{(i)} = c_{\gamma} \sqsubset c_{\alpha} =
d^{(i)}$ and $c^{(j)} = d^{(j)}$ for all $j < i$ (where $i$ is $|p| +
1$). Hence again $w' \ll w$.

Since $\ll$ is a well-ordering of $C^*$, there are no infinite $\ll$-infinite
descending chains in $C^*$. Thus, since every application of a rule
from $\rel{T}$ yields a $\ll$-preceding word, it follows that any
sequence of rewriting using $\rel{T}$ must terminate. Hence $\rel{T}$
is noetherian.
\end{proof}

\begin{lemma}
\label{lem:placticconfluent}
The rewriting system $(C,\rel{T})$ is confluent.
\end{lemma}

\begin{proof}
Let $w \in C^*$. Since $(C,\rel{T})$ is noetherian by
\fullref{Lemma}{lem:placticnoetherian}, applying $\rel{T}$ to $w$ will
always eventually yield an irreducible word. Let $w'$ and $w''$ be
irreducible words obtained from $w$. Suppose $w' =
c_{\alpha^{(1)}}c_{\alpha^{(2)}}\cdots c_{\alpha^{(k)}}$. Now, since
$w'$ is irreducible, it does not contain any subword forming a
left-hand side of a rule in $\rel{T}$. That is, there is no $i$ such
that $\alpha^{(i)} \not\succeq \alpha^{(i+1)}$. Equivalently,
$\alpha^{(i)} \succeq \alpha^{(i+1)}$ for all $i$. Thus
$\alpha^{(1)}\alpha^{(2)}\cdots \alpha^{(k)} = C(t')$ for some tableau
$t'$. But $t'$ must be the unique tableau with $t' =_{M_n}
\alpha^{(1)}\alpha^{(2)}\cdots \alpha^{(k)} =_{M_n} w'$. Similarly, if
$w'' = c_{\beta^{(1)}}c_{\beta^{(2)}}\cdots c_{\beta^{(l)}}$ then
$\beta^{(1)}\beta^{(2)}\cdots \beta^{(l)} = C(t'')$, where $t''$ is
the unique tableau with $t'' =_{M_n} \beta^{(1)}\beta^{(2)}\cdots
\beta^{(l)} =_{M_n} w''$. But since $w =_{M_n} w' =_{M_n} w''$, and
tableau form a cross-section of $M_n$, it follows that $t' =
t''$. Hence $k=l$ and $\alpha^{(i)} = \beta^{(i)}$ for all $i =
1,\ldots,k$, and so $w' = w''$. Hence rewriting an arbitrary word $w
\in C^*$ always terminates with a \emph{unique} irreducible word. Thus
the rewriting system $(C,\rel{T})$ is confluent.
\end{proof}

\fullref{Lemmata}{lem:placticnoetherian} and
\ref{lem:placticconfluent}, together with the finiteness of $\rel{T}$, yield the following result:

\begin{theorem}
$(C,\rel{T})$ is a finite complete rewriting system for the Plactic
  monoid $M_n$.
\end{theorem}

The following corollary is immediate \cite{squier_finiteness}:

\begin{corollary}
Every Plactic monoid has finite derivation type.
\end{corollary}

By a result originally proved by Anick in different form
\cite{anick_homology}, but also proved by various other authors (see
\cite{cohen_stringrewriting}):

\begin{corollary}
Every Plactic monoid is of type right and left $\mathrm{FP}_\infty$.
\end{corollary}

Now let $K$ be a field. Let $F = \{l - r : (l \imreduces r) \in
\rel{T}\} \subset K[C^*]$. Then the semigroup algebra $K[{M_n}]$ is isomorphic to the factor
algebra $K[C^*]/\!\gen{F}$ (where $\gen{F}$ is the ideal generated by
$F$) \cite[Proposition on p.~1]{heyworth_rewriting}. Since $(C,\rel{T})$ is a
finite complete rewriting system, $F$ is a finite Gr\"{o}bner--Shirshov basis
for $K[{M_n}]$ \cite[Theorem on p.~1]{heyworth_rewriting}. Furthermore, the order
$\ll$ defined in the proof of \fullref{Lemma}{lem:placticnoetherian}
corresponds in $K[C^*]$ to the degree-lexicographic
order. These remarks yield the following result:

\begin{theorem}
A Plactic algebra of arbitrary finite rank over an arbitrary field
admits a finite Gr\"{o}bner--Shirshov basis over $C$ with respect to
degree-lexicographic order.
\end{theorem}

\section{Biautomaticity}

The aim of this section is to prove that the Plactic monoid $M_n$ is
biautomatic. We will prove biautomaticity with respect to the usual
generating set $A$, but we will initially work with the generating set
$C$. The first step is to define a language of representatives over
$C$.

Let
\[
K = \bigl\{c_{\alpha^{(1)}}c_{\alpha^{(2)}}\cdots c_{\alpha^{(k)}} : k \in \nset \cup \{0\}, c_{\alpha^{(i)}} \in C, {\alpha^{(j)}} \succeq {\alpha^{(j+1)}} \text{ for all $j$}\bigr\}.
\]
Notice that for any
$c_{\alpha^{(1)}},c_{\alpha^{(2)}},\ldots,c_{\alpha^{(k)}} \in C$, we
have $c_{\alpha^{(1)}}c_{\alpha^{(2)}}\cdots c_{\alpha^{(k)}} \in K$
if and only if $\alpha^{(1)}\;\alpha^{(2)}\;\cdots \alpha^{(k)}$ is
the column reading of the corresponding tableau (that is,
$\alpha^{(1)}\;\alpha^{(2)}\;\cdots\;\alpha^{(k)} =
C(P(\alpha^{(1)}\alpha^{(2)}\cdots\alpha^{(k)}))$. Then $K$ is a
regular language over $C$, since an automaton need only store the
previously-read symbol in its state in order to check that
${\alpha^{(j)}} \succeq {\alpha^{(j+1)}}$. Actually, $K$ is the
language of normal forms for the rewriting system $(C,\rel{T})$
\cite[Lemma~2.1.3]{book_srs}. Duchamp \& Krob
\cite[\S~3.2]{duchamp_plactic} noted that this language $K$ is a
regular cross-section of the Plactic monoid, although their definition
of $K$ is rather different.

\subsection{Right-multiplication by transducer}

We will first of all prove that for any $\gamma \in A$ the relation
$K_{c_{\gamma}}$ is recognized by a finite transducer.

We imagine a transducer reading a pair of words
\[
(c_{\alpha^{(1)}}\cdots c_{\alpha^{(k)}},c_{\beta^{(1)}}\cdots c_{\beta^{(l)}}) \in K \times K
\]
\emph{from right to left}, with the aim of checking whether this pair
is in $K_{c_{\gamma}}$. It is easiest to describe the transducer as
reading symbols from the left tape and outputting symbols on the right
tape.  Essentially, the transducer will perform Schensted's algorithm
using the alphabet $C$ as a column representation of the tableau.

The transducer non-deterministically looks one symbol ahead (that is,
further left) on the input tape. In its state, it stores a symbol
$\eta$ from $A$ and a counter $m$ which can take any value from
$\{1,\ldots,n,\infty\}$. Initially, $\eta$ is set to be $\gamma$ and
$m$ is $1$, corresponding to the bottom row of the tableau. The idea
is that when $m \neq \infty$, the transducer is looking for the
correct column in which to insert $\eta$ in row $m$. Following
Schensted's algorithm, the transducer will know if it has found the
correct column $c_{\alpha^{(i)}}$ if the $m$-th symbol from the bottom
of $\alpha^{(i)}$ is greater than $\gamma$ and the $m$-th symbol from
the bottom of $\alpha^{(i-1)}$ is less than or equal to $\gamma$. The
crucial observation is that the transducer only needs a single
right-to-left pass because when a symbol $\eta$ is bumped, it is
inserted into the next row either in the same column or in the some
column further to the left, as was shown in
\fullref{Figure}{fig:bump}. When $m = \infty$, the transducer has
completed the algorithm and simply reads symbols from the input tape
and writes them on the output tape.

Initially, the transducer has $m = 1$, $\eta = \gamma$, and
non-deterministically knows $c_{\alpha^{(k)}}$. If the bottom symbol of
$\alpha^{(k)}$ is less than or equal to $\eta = \gamma$, then the
transducer outputs $c_{\gamma}$ before reading any input and then sets
$m = \infty$.

When reading a symbol $c_{\alpha^{(i)}}$, the transducer
non-deterministically knows $c_{\alpha^{(i-1)}}$ (or
non-deterministically guesses that it has reached
$c_{\alpha^{(1)}}$). As seen before, this is sufficient information to
check whether the symbol $\eta$ should be inserted into the column
$\alpha^{(i)}$ at row $m$ (bumping the $m$-th symbol from the bottom
of $\alpha^{(i)}$). If such an insertion and bump is carried out, $m$
is incremented by $1$ and $\eta$ replaced by the bumped symbol. The
transducer may have to carry out several such insertions and bumps
within the same column, but since there are only finitely many
possibilities for $c_{\alpha^{(i)}}$, $c_{\alpha^{(i-1)}}$, $m$, and
$\eta$, the result of carrying out all the necessary insertions and
bumps can be stored in a finite lookup table. The transition function
of the transducer can then be defined using this lookup table. Thus the
transducer can calculate the value of the resulting column $\beta$ and
output $c_\beta$. If no such insertion and bumping is carried out, the
transducer simply outputs $c_{\alpha^{(i)}}$.

Notice that when the transducer reads $c_{\alpha^{(i)}}$ and bumps
symbols it may increment $m$ to $|c_{\alpha^{(i)}}| + 1$. In this
case, the transducer must insert $\eta$ at the end of the $m$-th
row, which corresponds to finding the first (rightmost) symbol
$c_{\alpha^{(j)}}$ such that $|c_{\alpha^{(j-1)}}| \geq m$, and adding
$\eta$ to the top $\alpha^{(j)}$ to calculate the column $\beta$ and
output $c_\beta$. If the transducer reaches the leftmost end of the
input word without finding such an $c_{\alpha^{(j)}}$, the $m$-th row
is empty and so the symbol $\eta$ is added to the top of
$\alpha^{(1)}$. When a symbol is added to the top of some
$c_{\alpha^{(j)}}$, the transducer has completed the algorithm and
sets $m=\infty$.

Since it is recognized by a finite transducer, $L_{c_\gamma}$ is a rational
relation.

\subsection{Left-multiplication by transducer}

To prove that the relation ${}_{c_{\gamma}}K = \{(u,v) : u,v \in K,
c_\gamma u =_{M_n} v\}$ is recognized by a finite transducer whenever
$|\gamma| = 1$, we start with the following lemma, which is a
straightforward consequence of Schensted's algorithm:

\begin{lemma}
\label{lem:leftmultcolumns}
Let $\gamma \in A$ and let $\alpha = \alpha_p \cdots \alpha_1$ (where
$\alpha_i \in A$) be a column. Then
\begin{enumerate}
\item $\gamma > \alpha_p$ if and only if $P(\gamma\alpha)$ is a single
  column $\gamma\alpha_p\cdots\alpha_1$.
\item $r$ is minimal with $\gamma \leq \alpha_r$ if and only if
  $P(\gamma\alpha)$ has two columns: left column
  $\alpha_p\cdots\alpha_{r+1}\gamma\alpha_{r-1}\cdots\alpha_1$, and
  right column $\alpha_r$.
\end{enumerate}
\end{lemma}

Notice that if $c_\gamma c_\alpha$ is reducible with respect to the
rewriting system $(C,\rel{T})$, then $c_\gamma c_\alpha$ either
rewrites to a single symbol $c_{\gamma\alpha}$ with $\gamma\alpha
\succ \alpha$ or to a two-symbol word $c_{\alpha'}c_\eta$ with
$\alpha' \succeq \alpha$ and $\gamma \leq \eta$.

\begin{lemma}
\label{lem:leftmultcolumnssucc}
Let $\alpha = \alpha_p\cdots\alpha_1$ and $\beta =
\beta_q\cdots\beta_1$ be columns (where $\alpha_i,\beta_i \in A$) with
$\alpha \succeq \beta$. Let $i \in \{1,\ldots,p\}$, and let $\eta$ be
the left-hand column of $P(\alpha_i\beta)$. Then $\alpha \succeq
\eta$.
\end{lemma}

\begin{proof}
Since $\alpha \succeq \beta$, it follows that $p \geq q$ and $\alpha_j
\leq \beta_j$ for all $j \leq q$. We distinguish two cases:
\begin{enumerate}

\item Suppose $P(\alpha_i\beta)$ has two columns. Then $\eta$ has the
  form $\beta_q\cdots \beta_{r+1}\alpha_i\beta_{r-1}\cdots \beta_1$,
  where $r$ is minimal with $\alpha_i \leq \beta_r$. So $|\eta| =
  |\beta|$ and thus $|\alpha| \geq |\eta|$. Notice that $r \leq i$,
  since otherwise we would have $\alpha_i \leq \beta_i < \beta_r$,
  contradicting the minimality of $r$. Therefore $\alpha_r \leq \alpha_i$ and
  for all $j \leq q$ with $j \neq r$ we have $\alpha_j \leq
  \beta_j$. Thus $\alpha \succeq \eta$.

\item Suppose $P(\alpha_i\beta)$ has one column (namely $\eta$). By
  \fullref{Lemma}{lem:leftmultcolumns} we have $\alpha_i >
  \beta_q$. Since $\beta_q \geq \alpha_q$ and $\alpha$ is a column, we
  conclude $i > q$. Thus $p > q$. Therefore $\eta$ has the
  form $\alpha_i\beta_q\cdots\beta_1$ and $\alpha =
  \alpha_p\cdots\alpha_i\cdots\alpha_{q+1}\alpha_q\cdots\alpha_1$. Since
  $\alpha_{q+1}\leq\alpha_i$, it follows that $\alpha \succeq \eta$.\qedhere

\end{enumerate}
\end{proof}

\begin{lemma}
\label{lem:leftmultrewrite}
Let $\gamma \in A$ and $c_\alpha,c_\beta \in C$ with $\alpha \succeq \beta$.
\begin{enumerate}
\item If $c_\gamma c_\alpha c_\beta\imreduces c_{\alpha'}c_\eta c_\beta \imreduces c_{\alpha'}c_{\beta'}c_\zeta$, then $\alpha'\succeq\beta'$.
\item If $c_\gamma c_\alpha c_\beta\imreduces c_{\alpha'}c_\eta c_\beta \imreduces c_{\alpha'}c_{\beta'}$, then $\alpha'\succeq\beta'$.
\item If $c_\gamma c_\alpha c_\beta\imreduces c_{\alpha'}c_\beta$, then $\alpha'\succeq\beta$.
\end{enumerate}
\end{lemma}

\begin{proof}
\begin{enumerate}
\item From \fullref{Lemma}{lem:leftmultcolumns} and the remarks following it, we know that $\alpha' \succeq \alpha$ and $\eta$ is a letter from $\alpha$. Therefore, by \fullref{Lemma}{lem:leftmultcolumnssucc}, $\alpha \succeq \beta'$. Hence, since $\succeq$ is transitive, we have $\alpha' \succeq\beta'$.
\item The reasoning is the same as part~1.
\item From \fullref{Lemma}{lem:leftmultcolumns} and the remarks following it, we know that $\alpha' \succeq \alpha$. Since $\alpha \succeq \beta$ and $\succeq$ is transitive, we have $\alpha' \succeq \beta$.\qedhere
\end{enumerate}
\end{proof}

Let $c_{\alpha^{(1)}}\cdots c_{\alpha^{(k)}} \in K$. Recall that
${\alpha^{(1)}} \succeq \cdots \succeq {\alpha^{(k)}}$. Consider
rewriting the word $c_\gamma c_{\alpha^{(1)}}\cdots c_{\alpha^{(k)}}$
to normal form using rules in $\rel{T}$. Suppose the rewriting
proceeds as follows:
\begin{align*}
& c_\gamma c_{\alpha^{(1)}}c_{\alpha^{(2)}}\cdots c_{\alpha^{(k)}} \\
\imreduces{}& c_{\alpha'^{(1)}}c_{\gamma_1}c_{\alpha^{(2)}}\cdots c_{\alpha^{(k)}} \\
\imreduces{}& \ldots \\
\imreduces{}& c_{\alpha'^{(1)}}c_{\alpha'^{(2)}}\cdots c_{\alpha'^{(i)}}c_{\gamma_i}c_{\alpha^{(i+1)}}\cdots c_{\alpha^{(k)}}.
\end{align*}
By \fullref{Lemma}{lem:leftmultrewrite}, ${\alpha'^{(1)}} \succeq \cdots
\succeq {\alpha'^{(i)}}$. If $i = k$, then ${\alpha'^{(i)}} \succeq
{\gamma_i}$ by the definition of $\rel{T}$. Suppose rewriting
continues as follows:
\begin{align*}
& \ldots \\
\imreduces{}& c_{\alpha'^{(1)}}\cdots c_{\alpha'^{(j)}}c_{\gamma_j}c_{\alpha^{(j+1)}}c_{\alpha^{(j+2)}}\cdots c_{\alpha^{(k)}} \\
\imreduces{}& c_{\alpha'^{(1)}}\cdots c_{\alpha'^{(j)}}c_{\alpha'^{(j+1)}}c_{\alpha^{(j+2)}}\cdots c_{\alpha^{(k)}}.
\end{align*}
By \fullref{Lemma}{lem:leftmultrewrite}, $c\alpha'^{(i)} \succeq
\beta \succeq \alpha^{(j+2)}$. Hence rewriting $c_\gamma
c_{\alpha^{(1)}}\cdots c_{\alpha^{(k)}}$ to normal form requires only
a single left-to-right pass, which can be performed by a transducer:
it simply stores the symbol $c_{\gamma_i}$ in its state. Therefore the
relation ${}_{c_{\gamma}}K$ can be recognized by a transducer.

\subsection{Deducing biautomaticity}

Let $\rel{Q} \subseteq C^* \times A^*$ be the relation
\[
\bigl\{(c_{\alpha^{(1)}}c_{\alpha^{(2)}}\cdots c_{\alpha^{(k)}},\alpha^{(1)}\alpha^{(2)}\cdots\alpha^{(k)}) : k \in \nset\cup \{0\}, \text{each $\alpha^{(i)}$ is a column}\bigr\}.
\]
It is easy to see that $\rel{Q}$ is a rational relation. Let
\[
L = K \circ \rel{Q} = \bigl\{v \in A^* : (\exists u \in K)\bigl((u,v) \in \rel{Q}\bigr)\bigr\}.
\]
Then $L$ is a regular language over $A$ that maps onto $M_n$,
since the set of regular languages is closed under applying rational
relations. (In fact, $L$ is the set of column readings of tableaux,
but this is not important for us.) Then for any $\gamma \in A$,
\begin{align*}
(u,v) \in L_\gamma &\iff u \in L \land v \in L \land u\gamma =_{M_n} v\\
&\iff (\exists u',v' \in K)((u,u') \in \rel{Q} \land (v,v') \in \rel{Q} \land u'c_\gamma =_{M_n} v')\\
&\iff (\exists u',v' \in K)((u,u') \in \rel{Q} \land (v,v') \in \rel{Q} \land (u',v') \in K_{c_\gamma})\\
&\iff (u,v) \in \rel{Q}^{-1} \circ K_{c_\gamma} \circ \rel{Q}.
\end{align*}
Therefore, $L_\gamma$ is a rational relation. Now, if $(u,v) \in
L_\gamma$, then $|v| = |u|+1$ since $u\gamma =_{M_n} v$ and the
defining relations \eqref{eq:placticrel} preserve lengths of words.
By \fullref{Proposition}{prop:rationalbounded}, $L_\gamma\rpad$ and
$L_\gamma\lpad$ are regular.

Similarly, from the fact that ${}_{c_\gamma}K$ is a rational relation,
we deduce that ${}_\gamma L = \rel{Q}^{-1} \circ {}_{c_\gamma}K \circ
\rel{Q}$ is rational and thus, by
\fullref{Proposition}{prop:rationalbounded}, that ${}_\gamma L\rpad$
and ${}_\gamma L\lpad$ are regular.

\begin{theorem}
\label{thm:placticbiauto}
$(A,L)$ is a biautomatic structure for the Plactic monoid $M_n$.
\end{theorem}

\begin{corollary}
\label{corol:placticbiautoanygen}
Let $B$ be a generating set for the Plactic monoid $M_n$. Then $M_n$
admits a biautomatic structure over $B$.
\end{corollary}

\begin{proof}
Since each generator in $A$ admits no non-trivial decomposition in
$M_n$, it follows that every element of $A$ must also appear in
$B$. Hence $L$ is also a language over $B$. Let $b \in B$ and let
$u_1\cdots u_n \in A^*$ (where $u_i \in A$) be such that $b =_{M_n} u_1\cdots
u_n$. Then $L_b = L_{u_1} \circ L_{u_2} \circ \cdots \circ L_{u_n}$
and ${}_b L = {}_{u_1}L \circ {}_{u_2}L \circ \cdots \circ
{}_{u_n}L$. So $L_b\delta_R$, $L_b\delta_L$, ${}_b L\delta_R$, and
${}_b L\delta_R$ are all regular (see, for example,
\cite[Proposition~2.4]{hoffmann_notions}). Hence $(B,L)$ is a
biautomatic structure for $M_n$.
\end{proof}

\bibliography{semigroups,automaticsemigroups,languages,c_publications,\jobname}
\bibliographystyle{alphaabbrv}

\end{document}